\documentclass[11pt]{amsart}
\usepackage{amsmath,amsfonts,latexsym,graphicx,amssymb,url}
\usepackage{amsthm,txfonts,pifont,bbding,pxfonts,manfnt,mathrsfs}
\usepackage[dvipsnames]{xcolor}

\usepackage{amsmath,amsfonts,latexsym,graphicx,amssymb,url}
\usepackage{hyperref,pdfsync}
\usepackage{amsmath,txfonts,pifont,bbding,pxfonts,manfnt}
\usepackage[active]{srcltx}
\usepackage{wasysym,pstricks}
\usepackage{wrapfig, subfigure,graphicx}
\usepackage{hyperref}
\usepackage{pdfsync}
\usepackage{enumitem}

\usepackage{tikz}
\usetikzlibrary{arrows.meta}
\usetikzlibrary{datavisualization.formats.functions}
\usepackage{pgfplots}
\usetikzlibrary{shadows,shapes,positioning}

\newcommand{\E}{\mathbf{E}}
\renewcommand{\H}{\mathbf{H}}
\newcommand{\D}{\mathbf{D}}
\newcommand{\B}{\mathbf{B}}
\newcommand{\G}{\mathbf{G}}

\newcommand{\C}{{\mathbb C}}

\newcommand{\R}{{\mathbb R}} 

\newcommand{\e}{\varepsilon}

\renewcommand{\(}{\left(}
\renewcommand{\)}{\right)}

\renewcommand{\o}{\omega}

\newtheorem{prop}{Proposition}[section]
\newtheorem{remark}{Remark}[section]

\newtheorem{theorem}{Theorem}[section]

\newtheorem{lemma}[theorem]{Lemma}

\usepackage[margin=1in]{geometry}

\begin{document}

\title{Refraction laws in temporal media
}

\author{Cristian E. Guti\'errez}
\address[Cristian E. Guti\'errez]{Department of Mathematics, Temple University}
\email{gutierre@temple.edu}

\author{Eric Stachura}
\address[Eric Stachura]{Department of Mathematics, Kennesaw State University}
\email{estachur@kennesaw.edu}

\thanks{E. S. is the corresponding author. \today}

\begin{abstract}
We consider the time dependent Maxwell system in the sense of distributions in the context of temporal interfaces. Just as with spatial interfaces, electromagnetic waves at temporal interfaces scatter and create a transmitted and reflected wave. We provide a rigorous derivation of boundary conditions for the electric and magnetic fields at temporal interfaces with precise assumptions on the material parameters. In turn, we use this to obtain a general Snell's Law at such interfaces. From this, we obtain explicit formulas for the reflection and transmission coefficients. Unlike previous works, we do not make any simplifying ansatz on the solution to the Maxwell system, nor do we assume that the fields are smooth. We also consider material parameters which are not necessarily constant on either side of the temporal interface.
\end{abstract}

\maketitle

\tableofcontents

\setcounter{equation}{0}
\section{Introduction}
The area of time varying materials has received great attention lately, with the idea that time can be used to acheive additional control of electromagnetic waves. Many applications, especially related to temporal modulations of metamaterials \cite{mostafa2024temporal}, have been seen so far, and include frequency conversion, wave amplification, and the foundations of photonic time crystals, among many others. Photonic time crystals are systems which undergo periodic temporal variations while maintaining spatial uniformity \cite{lustig2023photonic}.

While the applications are vast, the mathematical analysis related to the Maxwell equations, in particular, in such time varying media is lacking in comparison (in no small part due to the mathematical challenges of time varying material parameters). For mathematical results in the time domain with time varying material parameters, we mention in particular \cite{farago2012numerical} for numerical investigations using a so-called Magnus expansion and \cite[Section 5]{stachura2024quantitative} which considers weak solutions for the time dependent Maxwell system with both space and time varying permittivity. We also mention various works that consider time dependence through temperature dependent material parameters, see in particular \cite{alam2022coupled, yin2004regularity, yin1994global} and the references therein. 

In this paper, we focus on the particular case of temporal interfaces. 
A temporal interface occurs \cite{engheta2023four} when a material parameter is changed in time rapidly while the optical wave is present in the material. One way this can be obtained is through a strong optical nonliearity and an ultrafast (on the other of 5-10 femtoseconds) laser pulse \cite{tirole2024second}.

At a temporal interface, foward and backward waves \emph{in time} (waves that propogate along the same and opposite directions of the original wave) are known to result \cite{morgenthaler1958velocity, mendoncca2002time}. For this temporal interface, the wavelength for the forward and backward waves are the same as the original wave, but the frequency changes due to the change in wave velocity. Moreover, at the temporal interface (unlike the classical spatial interface), the energy needed to change the material can affect the electromagnetic wave energy, so energy is not conserved (it can increase or decrease). We will derive these results from a rigorous mathematical perspective, and in particular, quantify the deviation from conservation of energy--see equations (\ref{conservation_of_energy1}) and (\ref{conservation_of_energy2}). This depends explicitly on the material parameters.

\begin{figure}[h!]
\centering
\begin{tikzpicture}
\begin{axis}[
  width=0.5\linewidth,
  ymin=0,
  ymax=5,
  yticklabel=\empty,
  ylabel={$(x,y,z)$},
  ylabel style={rotate=-90},
  xmin=0,
  xmax=5,
  xtick style={draw=none},
  ytick style={draw=none},
  xticklabel=\empty,
  axis x line*=bottom,
  axis y line*=left,
  xlabel= $t$,
  extra x ticks = {2.5},
  extra x tick labels={$t_0$}]
\addplot[thick,domain=0:5, black] coordinates {(2.5,0)(2.5, 5)};
\addplot[draw=none] coordinates {(1,1)};
\node[] at (axis cs: 1,4) {$n_1(t)$};
\node[] at (axis cs: 3, 4) {$n_2(t)$};
\draw[thick, black, decorate,decoration={expanding waves,angle=7}] (axis cs: 0.5, 0.5)   -- (axis cs: 2.5, 2.5);
\draw[thick, arrows={-Stealth[scale=1.5]}] (axis cs: 0.5, 0.5) -- (axis cs: 2, 2);
\draw[thick, black, decorate,decoration={expanding waves,angle=7}] (axis cs: 2.5, 2.5)   -- (axis cs: 4.5, 4.5);
\draw[thick, arrows={-Stealth[scale=1.5]}] (axis cs: 2.5, 2.5) -- (axis cs: 4, 4);
\draw[thick, black, decorate,decoration={expanding waves,angle=7}] (axis cs: 5.5, 3.5)   -- (axis cs: 3.5, 1.5);
\draw[thick, arrows={-Stealth[scale=1.5]}] (axis cs: 5.5, 3.5)   -- (axis cs: 4, 2);
\end{axis}
\end{tikzpicture}
%
%
\caption{A temporal interface where the time varying refractive index is $n_1(t)$ on one side of the interface and $n_2(t)$ on the other side.}
\label{default}
\end{figure}
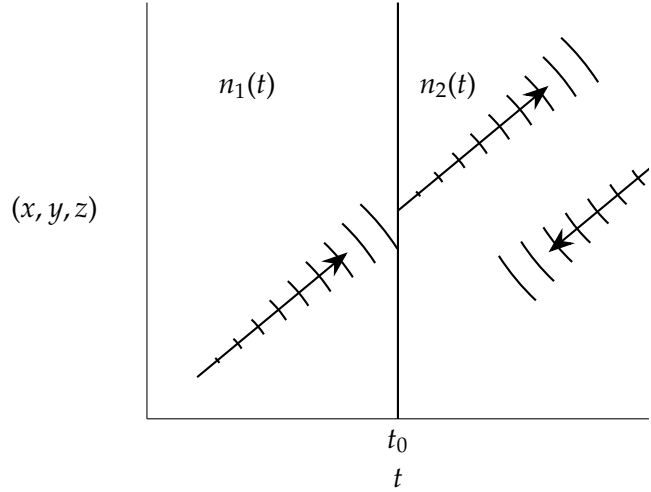

The objective of this paper is to derive the Snell law at temporal interfaces from fundamental principles. To achieve this, we will utilize the Maxwell equations. Given that the waves being considered are generally discontinuous, it is advantageous to approach this derivation from the perspective of distributions or generalized functions, thereby facilitating a comprehensive understanding of the Maxwell system.
A similar distributional analysis of Maxwell equations in the context of spatial metasurfaces has been done in \cite{2025-gutierrez-sabra:Maxwelleqsandgeneralizedsnelllaw}. 

The outline of this paper is as follows. In Section \ref{sec:Maxwell_distributions}, we analyze the Maxwell system in sense of space-time distributions. Under certain assumptions on fields which are discontinuous at a certain time, we derive a formula for the distributional time derivative of such a field, see equation (\ref{eq:formula for the distributional derivative wrt t}). Then in Section \ref{sec:boundary conditions}, under particular assumptions on the material parameters $\e, \mu$, we provide a rigorous derivation of the boundary conditions satisfied by the electric and magnetic fields at a temporal interface. Again, the time dependent Maxwell system is understood in the distributional sense here. In Section \ref{sec:SnellLaw} we use these boundary conditions to derive a fully rigorous Snell's Law for reflection and refraction at a temporal interface in the case when $\e, \mu$ vary in time but not in space, Proposition \ref{prop:snell law}. From here, in Section \ref{subsec:epsilon mu jump constants}, we consider the simplified case when $\e, \mu$ jump from one value to another (a common assumption in practice) and rigorously obtain formulas for the amplitudes, Proposition \ref{prop:formulas for the amplitudes}, and for the reflection and transmission coefficients, Section \ref{subsec:reflection and transmission coefficients}. We also note that the amplitudes of the incident, transmitted, and reflected waves are generally not arbitrary, but rather subject to the divergence equations in the Maxwell system--see Section \ref{sec:divergence}. 

\setcounter{equation}{0}
\section{Maxwell equations in the sense of distributions}\label{sec:Maxwell_distributions}

Recall the Maxwell system of equations
\begin{align}
&\nabla \cdot \D = \rho, \label{divergence E zero}\tag{M.1}\\
&\nabla \cdot \B = 0, \label{divergence B zero}\tag{M.2}\\
&\nabla \times \E = -\dfrac{\partial \B}{\partial t}\label{Faraday law}\tag{M.3}\\
&\nabla \times \H= \mathbf{J}+\dfrac{\partial \D}{\partial t}. \label{Ampere Maxwell}\tag{M.4}
\end{align}
where $\E(x,t)$ denotes the electric field, $\H(x,t)$ the magnetic field, $\D(x,t)$ the electric flux density and $\B(x,t)$ the magnetic flux density. 
We assume the charge density $\rho=0$ and the current density $\mathbf{J}=0$, and
consider the constitutive equations given by 
\begin{equation}\label{eq:constitutive}
\D=\e\,\E,\qquad \B=\mu\,\H
\end{equation}
where we assume that $\e=\e(x,t)$ and $\mu=\mu(x,t)$ are positive functions of $x\in \R^3,t\in \R$, $t\neq t_0$ for some $t_0$, that satisfy appropriate piecewise smoothness assumptions that will be described in a moment.

We begin recalling the notions needed to analyze the Maxwell system in distributional sense.
Let $\Omega\subset \R^3$ be an open domain, that could be the whole space, and let $(a,b)\subset \R$ be an open interval that could be infinite. 
A generalized function or distribution defined in $\Omega\times (a,b)$ is a complex-valued continuous linear functional defined on the class of test functions denoted by $\mathcal D(\Omega\times (a,b))=C_0^\infty(\Omega\times (a,b))$, that is, the class of functions are infinitely differentiable in $\Omega\times (a,b)$ and have compact support in $\Omega\times (a,b)$. 
More precisely, $g$ is a distribution in $\Omega\times (a,b)$ if $g:\mathcal D(\Omega\times (a,b))\to \C$ is a linear function such that for each compact $K\subset \Omega\times (a,b)$ there exist constants $C$ and an integer $k$ such that 
\begin{equation*}\label{eq:distribution estimate}
|g(\phi)|\leq C\,\sum_{|\alpha|\leq k}\sup_{x\in K}|D^\alpha\phi(x)|
\end{equation*}
for each $\phi\in \mathcal D(\Omega\times (a,b))$.
%
As customary, $\mathcal D’(\Omega\times (a,b))$ denotes the class of distributions in $\Omega\times (a,b)$. 

Numerous references are available for distributions; however, we only cite the seminal work by L. Schwartz, \cite{schwartz:theoriedesdistributions}, and also \cite{folland1999real}. 

If $g\in \mathcal D'(\Omega\times (a,b))$, then $\langle g,\phi \rangle$ denotes the value of the distribution $g$ on the test function $\phi\in \mathcal D(\Omega\times (a,b))$.
If $g$ is a locally integrable function in $\Omega\times (a,b)$, then $g$ gives rise to a distribution defined by 
\[
\langle g,\phi\rangle=\int_{\Omega\times (a,b)} g(x,t)\,\phi(x,t)\,dxdt,
\]
for each $\phi\in \mathcal D(\Omega\times (a,b))$.

We say that ${\bf G}=(G_1,G_2,G_3)$ is a vector valued distribution in $\Omega\times (a,b)$ if each component $G_i\in \mathcal D'(\Omega\times (a,b))$, $1\leq i\leq 3$.
The divergence of ${\bf G}$ with respect to $x$ is the scalar distribution defined by
\begin{equation}\label{form:divergence}
\langle \nabla \cdot {\bf G}, \phi\rangle
=-\sum_{i=1}^3 \langle G_i,\partial_{x_i} \phi\rangle,
\end{equation}
and the curl of ${\bf G}$ is the vector valued distribution in $\Omega$ defined by
\begin{equation}\label{form:curl bis}
\langle \nabla \times {\bf G},\phi\rangle=\left(\langle G_2,\phi_{x_3}\rangle-\langle G_3,\phi_{x_2}\rangle\right){\bf i}-\left(\langle G_1,\phi_{x_3}\rangle-\langle G_3,\phi_{x_1}\rangle\right){\bf j}+\left(\langle G_1,\phi_{x_2}\rangle-\langle G_2,\phi_{x_1}\rangle\right){\bf k}.
\end{equation}
Then it follows that
\begin{equation}\label{eq:div of curl zero}
\nabla \cdot (\nabla\times {\bf G})=0,
\end{equation} 
in the sense of distributions.
When the distribution ${\bf G}=(G_1,G_2,G_3)$ is given by a locally integrable in $\Omega\times (a,b)$ we obtain from \eqref{form:divergence}, and \eqref{form:curl bis} that 
\begin{equation*}
\langle \nabla \cdot {\bf G},\phi\rangle=-\int_{\Omega\times (a,b)}{\bf G}\cdot \nabla_x \phi\,dxdt,\qquad
\langle \nabla \times {\bf G},\phi\rangle=\int_{\Omega\times (a,b)} {\bf G}\times \nabla_x \phi\,dxdt.
\end{equation*}
The derivative of $\G$ with respect to $t$ in the sense of distributions is by definition the distribution $\dfrac{\partial \G}{\partial t}$ defined by
\[
\left\langle \dfrac{\partial \G}{\partial t},\phi \right\rangle
=
-\left\langle \G, \dfrac{\partial \phi}{\partial t}\ \right\rangle,
\]
for each $\phi\in \mathcal D(\Omega\times (a,b))$.

Therefore, the fields $\D,\B,\E$ and $\H$ solve the Maxwell system in the sense of distributions if 
$\D,\B,\E$ and $\H$ are vector valued distributions in $\Omega\times (a,b)$ that satisfy the equations \eqref{divergence E zero}, \eqref{divergence B zero}, \eqref{Faraday law}, and \eqref{Ampere Maxwell} in the sense of distributions.

\subsection{General formulas}\label{sec:general_formulas}
Given that our application to temporal interfaces needs the consideration of distributions represented by discontinuous functions, we require a comprehensive set of representation formulas that will be utilized in Section \ref{sec:boundary conditions} to prove the boundary conditions.

Let $\Omega\subset \R^3$ be an open domain, which could be all of $\R^3$, $t_0\in (a,b)\subset \R$, $(a,b)$ could be infinite, and we have a field $\G(x,t)$ defined for $x\in \Omega$ and $t\neq t_0$ as follows:
\begin{equation*}
\G(x,t)
=
\begin{cases}
\G_-(x,t) &\text{for $x\in \Omega$ and $t<t_0$}\\
\G_+(x,t) &\text{for $x\in \Omega$ and $t>t_0$}
\end{cases}
\end{equation*}
and satisfying
%
\begin{enumerate}
\item\label{eq:G is C1} $\G_-\in C^1\(\Omega\times (a,t_0)\)$,  $\G_+\in C^1\(\Omega\times (t_0,b)\)$, and $\G\in L^1_{\text{loc}}\(\Omega\times (a,b)\)$\footnote{As usual, $L^1_{\text{loc}}\(\Omega\times (a,t_0)\)$ denotes the class of complex valued functions that are locally Lebesgue integrable in $\Omega\times (a,t_0)$.};
\item\label{eq:limits are finite} the limits
\[
\lim_{t\to t_0^-}\G_-(x,t):=\G_-(x,t_0),\qquad \lim_{t\to t_0^+}\G_+(x,t):=\G_+(x,t_0)\]
exist and are finite for all $x\in \Omega$;
\item\label{eq:derivatives if G wrt t are integrable} $\dfrac{\partial \G}{\partial t}$ and $\nabla_x \G$ are locally integrable in $\Omega\times (a,b)$.
\end{enumerate}
Given a test function $\phi\in C_0^\infty \(\Omega\times (a,b)\)$ the linear functional
\[
\langle \G,\phi\rangle=\int_\Omega \int_a^{t_0} \G_-(x,t)\phi(x,t)\,dxdt
+\int_\Omega \int_{t_0}^b \G_+(x,t)\phi(x,t)\,dxdt
\]
defines a distribution in $\Omega\times (a,b)$.
Define 
\[
[[\G]]=\G_+(x,t_0)-\G_-(x,t_0)\]
to be the jump of the field at $t=t_0$.
\begin{prop}
Under the assumptions \ref{eq:G is C1}, \ref{eq:limits are finite}, and \ref{eq:derivatives if G wrt t are integrable} above on $\G$ we then have the following formula
\begin{equation}\label{eq:formula for the distributional derivative wrt t}
\left\langle \dfrac{\partial \G}{\partial t},\phi \right\rangle
=
\int_\Omega [[ \G(x,t_0)]]\,\phi(x,t_0)\,dx 
+
\int_\Omega \int_a^{t_0} \dfrac{\partial \G_-}{\partial t}(x,t)\,\phi(x,t)\,dx\,dt
+
\int_\Omega \int_{t_0}^b \dfrac{\partial \G_+}{\partial t}(x,t)\,\phi(x,t)\,dx\,dt,
\end{equation}
for all $\phi \in \mathcal D(\Omega\times (a,b))$.
\end{prop}
\begin{proof}
Let $\phi \in \mathcal D(\Omega\times (a,b))$.
We have 
\begin{align*}
\left\langle \dfrac{\partial \G}{\partial t},\phi \right\rangle
&=
-
\int_\Omega \int_a^{t_0} \G_-(x,t)\dfrac{\partial \phi}{\partial t}(x,t)\,dxdt
-\int_\Omega \int_{t_0}^b \G_+(x,t)\dfrac{\partial \phi}{\partial t}(x,t)\,dxdt\\
&=
-
\int_\Omega \(\int_a^{t_0} \G_-(x,t)\dfrac{\partial \phi}{\partial t}(x,t)\,dt
+
\int_{t_0}^b \G_+(x,t)\dfrac{\partial \phi}{\partial t}(x,t)\,dt\)\,dx
\end{align*}
Integrating by parts with respect to $t$ yields
\[
\int_a^{t_0} \G_-(x,t)\dfrac{\partial \phi}{\partial t}(x,t)\,dt
=
\left. \G_-(x,t) \,\phi(x,t)\right|_{t=a}^{t_0^-}
-
\int_a^{t_0} \dfrac{\partial \G_-(x,t)}{\partial t}\,\phi(x,t)\,dt
\]
and
\[
\int_{t_0}^b \G_+(x,t)\dfrac{\partial \phi}{\partial t}(x,t)\,dt
=
\left. \G_+(x,t) \,\phi(x,t)\right|_{t_0^-}^{t=b}
-
\int_{t_0}^b \dfrac{\partial \G_+(x,t)}{\partial t}\,\phi(x,t)\,dt
\]
for each $x\in \Omega$. Inserting these into the first equation yields the desired formula.

\end{proof}

From assumptions \ref{eq:G is C1}, \ref{eq:limits are finite}, and \ref{eq:derivatives if G wrt t are integrable}, it follows that  
\eqref{form:divergence} takes the form 
\begin{align*}
\left\langle \nabla_x \cdot \G,\phi \right\rangle
&=
-
\sum_{i=1}^3
\int_\Omega \int_a^b G_i(x,t)\,\dfrac{\partial \phi(x,t)}{\partial x_i}\,dxdt\\
&=
-
\sum_{i=1}^3
\(\int_\Omega \int_a^{t_0} G_i(x,t)\,\dfrac{\partial \phi(x,t)}{\partial x_i}\,dxdt
-\int_\Omega \int_{t_0}^b G_i(x,t)\,\dfrac{\partial \phi(x,t)}{\partial x_i}\,dxdt\)\\
&=-
\int_\Omega \int_a^{t_0} \G_-(x,t)\cdot \nabla_x \phi(x,t) \,dxdt
+\int_\Omega \int_{t_0}^b \G_+(x,t)\cdot \nabla_x \phi(x,t)\,dxdt\\
&=-
\int_a^{t_0}\(\int_\Omega  \G_-(x,t)\cdot \nabla_x \phi(x,t) \,dx\)dt
-\int_{t_0}^b\(\int_\Omega  \G_+(x,t)\cdot \nabla_x \phi(x,t)\,dx\)dt.
\end{align*}
Since $\G_\pm$ are differentiable in $x$ for $t\neq t_0$, we have
\begin{align*}
\int_\Omega  \G_\pm(x,t)\cdot \nabla_x \phi(x,t) \,dx
&=
\int_\Omega \(\nabla_x\cdot \(\phi\, \G_\pm\)-\phi\,\nabla_x\cdot \G_\pm\)\,dx\\
&=
\int_{\partial \Omega} \phi(x,t)\, \G_\pm(x,t)\cdot {\bf{n}} (x)\,d\sigma(x)-\int_\Omega \phi(x,t)\,\nabla_x\cdot \G_\pm(x,t)\,dx\\
&=
-\int_\Omega \phi(x,t)\,\nabla_x\cdot \G_\pm(x,t)\,dx
\end{align*}
by the divergence Theorem and since $\phi$ has compact support in $\Omega\times (a,b)$. Above, $d\sigma$ denotes surface measure. If $\Omega$ is unbounded the boundary term above also vanishes since $\phi$ has compact support. Hence we obtain the formula
\[
\left\langle \nabla_x \cdot \G,\phi \right\rangle
=
-\(\int_\Omega \int_a^{t_0} \phi(x,t)\,\nabla_x\cdot \G_-(x,t)\,dx
+\int_\Omega \int_{t_0}^b\phi(x,t)\,\nabla_x\cdot \G_+(x,t)\,dx\)
\]
for all $\phi\in C_0^\infty \(\Omega\times (a,b)\)$.
If $\G$ satisfies $\nabla_x\cdot \G=0$ in the sense of distributions, then  $\left\langle \nabla_x \cdot \G,\phi \right\rangle=0$, and in particular  
\[
\int_\Omega \int_a^{t_0} \phi(x,t)\,\nabla_x\cdot \G_-(x,t)\,dx=0\]
for all $\phi\in C_0^\infty \(\Omega\times (a,t_0)\)$ and so $\nabla_x\cdot \G_-(x,t)=0$ for all $x\in \Omega$ and $a<t<t_0$.
Similarly, we get $\nabla_x\cdot \G_+(x,t)=0$ for all $x\in \Omega$ and $t_0<t<b$.

\setcounter{equation}{0}
\section{Boundary conditions}\label{sec:boundary conditions}

In this section, we assume the material parameters $\varepsilon(x,t), \mu(x,t)$ are defined for $x\in \Omega$ and $t\neq t_0$ as follows:
\begin{equation*}
\e(x,t)
=
\begin{cases}
\e_-(x,t) & \text{for $t<t_0$}\\
\e_+(x,t) & \text{for $t>t_0$}
\end{cases},
\qquad
\mu(x,t)
=
\begin{cases}
\mu_-(x,t) & \text{for $t<t_0$}\\
\mu_+(x,t) & \text{for $t>t_0$}
\end{cases};
\end{equation*}
and satisfy 
\begin{enumerate}
  \item[\hypertarget{H1}{H1}] the limits 
  \begin{align*}
  \lim_{t\to t_0^-} \e_-(x,t)&:=\e_-(x,t_0),\quad \lim_{t\to t_0^+} \e_+(x,t):=\e_+(x,t_0),\\ \quad \lim_{t\to t_0^-} \mu_-(x,t)&:=\mu_-(x,t_0),\quad \lim_{t\to t_0^+} \mu_+(x,t):=\mu_+(x,t_0)  
  \end{align*}
  all exist and are finite.
  
  \item[\hypertarget{H2}{H2}] $\varepsilon$ and $\mu$ are differentiable in $t$ for $t\neq t_0$ with bounded derivatives.
  
  \item[\hypertarget{H3}{H3}] $\varepsilon$ and $\mu$ are differentiable in $x$ for all $t\neq t_0$ with derivatives in $x$ bounded in $\Omega$.
\end{enumerate}
%


\subsection{Boundary condition for the electric field}\label{subsec:boundary condition for the electric field}
Suppose $\e=\e(x,t)$ satisfies  \hyperlink{H1}{H1}, \hyperlink{H2}{H2}, and \hyperlink{H3}{H3}
 and we interpret the Maxwell equation resulting from \eqref{Ampere Maxwell} and \eqref{eq:constitutive}
\begin{equation}\label{eq:maxwell equation Faraday}
\nabla \times \H=\dfrac{1}{c} \dfrac{\partial }{\partial t}\(\e \E\)
\end{equation}
in the sense of distributions, i.e., if $\phi=\phi(x,t)$ is a smooth function with compact support in $\Omega\times (a,b)$, then
\[
\langle \nabla \times \H, \phi \rangle = \dfrac{1}{c} \left\langle \dfrac{\partial }{\partial t}\(\e \E\),\phi \right\rangle.
\]
From \hyperlink{H1}{H1} 
we set 
\[
\lim_{t\to t_0^-} \e(x,t)=\e_-(x,t_0),\qquad \lim_{t\to t_0^+} \e(x,t)=\e_+(x,t_0).
\]
We assume the field $\E$ has the form
\begin{equation*}
\E=
\begin{cases}
\E_- & \text{for $t<t_0$}\\
\E_+ & \text{for $t>t_0$}
\end{cases}
\end{equation*}
and satisfies conditions \ref{eq:G is C1}, \ref{eq:limits are finite}, and \ref{eq:derivatives if G wrt t are integrable} from Section \ref{sec:general_formulas}; we set
\[
\lim_{t\to t_0^-} \E(x,t)=\E_-(x,t_0),\qquad \lim_{t\to t_0^+} \E(x,t)=\E_+(x,t_0).
\]
Next let 
\[
\G
=
\begin{cases}
\G_-(x,t)=\e_-(x,t)\E_-(x,t) & \text{for $x\in \Omega$ and $a<t<t_0$}\\
\G_+(x,t)=\e_+(x,t)\E_+(x,t) & \text{for $x\in \Omega$ and $t_0<t<b$}.
\end{cases}
\]
Since $\e$ satisfies \hyperlink{H1}{H1}, \hyperlink{H2}{H2}, and \hyperlink{H3}{H3}, it follows that 
$\G=\e\E$ satisfies \ref{eq:G is C1}, \ref{eq:limits are finite}, and \ref{eq:derivatives if G wrt t are integrable} and so we can apply \eqref{eq:formula for the distributional derivative wrt t} to obtain
\begin{align}\label{eq:formula for dt of epsilon E}
\left\langle \dfrac{\partial }{\partial t}\(\e \E\),\phi \right\rangle
&=
\int_\Omega [[ \e (x,t_0)\E(x,t_0)]]\,\phi(x,t_0)\,dx \\
&\qquad +
\int_\Omega \int_a^{t_0} \dfrac{\partial \(\e_-(x,t)\, \E(x,t)\)}{\partial t}\,\phi(x,t)\,dt\,dx
+
\int_\Omega \int_{t_0}^b \dfrac{\partial \(\e_+(x,t)\, \E(x,t)\)}{\partial t}\,\phi(x,t)\,dt\,dx,\notag
\end{align}
with
\[
[[ \e (x,t_0)\E(x,t_0)]]
=
 \e_+(x,t_0)\, \E_+(x,t_0)- \e_-(x,t_0)\, \E_-(x,t_0).
\]

Assuming also that $\H$ satisfies conditions \ref{eq:G is C1}, \ref{eq:limits are finite}, and \ref{eq:derivatives if G wrt t are integrable}, we will show that the Maxwell equation \eqref{eq:maxwell equation Faraday} is satisfied pointwise, that is,
\[
\dfrac{1}{c}\dfrac{\partial \(\e(x,t)\, \E(x,t)\)}{\partial t}=\nabla \times\H(x,t)
\]
for all $x$ and $t\neq t_0$.
Let us prove this claim. 
Let $\phi\in C_0^\infty\(\Omega\times (a,t_0)\)$. From assumptions \ref{eq:G is C1},  and \ref{eq:derivatives if G wrt t are integrable}, $\H(x,t)$ is differentiable both in $x$ and $t$ for $t\neq t_0$, and it follows that 
\begin{align*}
\left\langle \nabla \times\H,\phi \right\rangle
&=
\int_\Omega \int_a^b \H(x,t)\times \nabla\phi(x,t)\,dxdt
=
\int_\Omega \int_a^{t_0} \H(x,t)\times \nabla\phi(x,t)\,dxdt\\
&=
\int_\Omega \int_a^{t_0} \phi(x,t)\,\nabla \times \H(x,t)\,dxdt\\
&= \dfrac{1}{c}\,\int_\Omega \int_a^{t_0} \dfrac{\partial \(\e_-(x,t)\, \E(x,t)\)}{\partial t}\,\phi(x,t)\,dt\,dx,\quad \text{from \eqref{eq:formula for dt of epsilon E} since $\phi$ has compact support in $\Omega\times (a,t_0)$}
\end{align*}
hence 
\[
\nabla \times \H(x,t)=\dfrac{1}{c}\dfrac{\partial \(\e_-(x,t)\, \E(x,t)\)}{\partial t}\quad \text{for all $x\in \Omega$ and $a<t<t_0$}.\]
Similarly, taking $\phi \in C_0^\infty\(\Omega\times (t_0,b)\)$, we get the same point-wise identity with $\e_-$ replaced by $\e_+$ for $x\in \Omega$ and $t_0<t<b$.

Therefore, from the pointwise Maxwell equation \eqref{eq:maxwell equation Faraday} we get
\[
\left\langle \dfrac{\partial }{\partial t}\(\e \E\),\phi \right\rangle
=
c \langle \nabla \times\H,\phi\rangle
=
c \int_\Omega \int_a^b \nabla \times\H(x,t) \phi(x,t)dxdt
=
  \int_\Omega \int_a^b \dfrac{\partial \(\e(x,t)\, \E(x,t)\)}{\partial t} dx dt,
  \]
  and so from \eqref{eq:formula for dt of epsilon E} we get 
  \[
  \int_\Omega [[ \e (x,t_0)\E(x,t_0)]]\,\phi(x,t_0)\,dx=0
  \]
  for all $\phi$. As a result, we obtain the boundary condition for the electric field:
\begin{equation}\label{eq:jump of epsilon E is zero}
[[ \e (x,t_0)\E(x,t_0)]]=0
\end{equation}
for all $x$.

 \subsection{Boundary condition for the magnetic field} 
We next consider the Maxwell equation resulting from \eqref{Faraday law} and \eqref{eq:constitutive}
\begin{equation}\label{eq:maxwell equation Ampere}
\nabla \times \E=-\dfrac{1}{c} \dfrac{\partial }{\partial t}\(\mu \H\).
\end{equation}  
in the sense of distributions, i.e., if $\phi=\phi(x,t)\in C_0^\infty\(\Omega\times (a,b)\)$, then
\[
\langle \nabla \times \E, \phi \rangle = -\dfrac{1}{c} \left\langle \dfrac{\partial }{\partial t}\(\mu \H\),\phi \right\rangle.
\]
We will proceed in the same way as in the previous section to get a boundary condition for $\H$.
Since $\mu$ satisfies \hyperlink{H1}{H1}, \hyperlink{H2}{H2}, and \hyperlink{H3}{H3}
 we set  
\[
\lim_{t\to t_0^-} \mu(x,t)=\mu_-(x,t_0),\qquad \lim_{t\to t_0^+} \mu(x,t)=\mu_+(x,t_0).
\]
Assuming also that $\H$ satisfies \ref{eq:G is C1}, \ref{eq:limits are finite}, and \ref{eq:derivatives if G wrt t are integrable} from Section \ref{sec:general_formulas},
we set
\[
\lim_{t\to t_0^-} \H(x,t)=\H_-(x,t_0),\qquad \lim_{t\to t_0^+} \H(x,t)=\H_+(x,t_0).
\]
If we let 
\[
\G
=
\begin{cases}
\G_-(x,t)=\mu_-(x,t)\H(x,t) & \text{for $x\in \Omega$ and $a<t<t_0$}\\
\G_+(x,t)=\mu_+(x,t)\H(x,t) & \text{for $x\in \Omega$ and $t_0<t<b$},
\end{cases}
\]
then $\G=\mu\H$ satisfies \ref{eq:G is C1}, \ref{eq:limits are finite}, and \ref{eq:derivatives if G wrt t are integrable} and so
applying \eqref{eq:formula for the distributional derivative wrt t} yields
\begin{equation}\label{eq:formula for dt of epsilon H}
\left\langle \dfrac{\partial }{\partial t}\(\mu \H\),\phi \right\rangle
=
\int_\Omega [[ \mu (x,t_0)\H(x,t_0)]]\,\phi(x,t_0)\,dx 
+
\int_\Omega \int_a^b \dfrac{\partial \(\mu(x,t)\, \H(x,t)\)}{\partial t}\,\phi(x,t)\,dt\,dx,
\end{equation}
where
\[
[[ \mu (x,t_0)\H(x,t_0)]]
=
 \mu_+(x,t_0)\, \H_+(x,t_0)- \mu_-(x,t_0)\, \H_-(x,t_0),
\]

On the other hand, 
assuming $\E$ satisfies \ref{eq:G is C1}, \ref{eq:limits are finite}, and \ref{eq:derivatives if G wrt t are integrable},
%
proceeding as in Section \ref{subsec:boundary condition for the electric field}, it follows that the Maxwell equation \eqref{eq:maxwell equation Ampere} is satisfied pointwise, that is,
\[
-\dfrac{1}{c}\dfrac{\partial \(\mu(x,t)\, \H(x,t)\)}{\partial t}=\nabla \times\E(x,t)
\]
for all $x$ and $t\neq t_0$.
Therefore, we obtain
\[
\left\langle \dfrac{\partial }{\partial t}\(\mu \H\),\phi \right\rangle
=
-c \langle \nabla \times\E,\phi\rangle
=
-c \int_\Omega \int_a^b \nabla \times\E(x,t) \phi(x,t)dxdt
=
  \int_\Omega \int_a^b \dfrac{\partial \(\mu(x,t)\, \H(x,t)\)}{\partial t} dx dt,
  \]
  and so from \eqref{eq:formula for dt of epsilon H} we get 
  \[
  \int_\Omega [[ \mu (x,t_0)\H(x,t_0)]]\,\phi(x,t_0)\,dx=0
  \]
  for all $\phi$. This yields the boundary condition for the magnetic field
   \begin{equation}\label{eq:jump of mu H is zero}
  [[ \mu (x,t_0)\H(x,t_0)]]=0.
  \end{equation}


\setcounter{equation}{0}
\section{Snell's law for time varying media}\label{sec:SnellLaw}
In this section, we assume that the material parameters $\varepsilon$ and $\mu$ vary in time but not in space. Suppose that time $t=t_0$ corresponds to a temporal interface, and that $\e_{\pm}, \mu_{\pm}$ are as in the previous section, now independent of space. Define the velocities
\[
v_-(t)=\dfrac{1}{\sqrt{\e_-(t)\mu_-(t)}}, \qquad t<t_0,
\]
and
\[
v_+(t)=\dfrac{1}{\sqrt{\e_+(t)\mu_+(t)}}, \qquad t>t_0.
\]
We further assume that $v_{\pm}'(t)$ are integrable in time so that condition \eqref{eq:derivatives if G wrt t are integrable} is verified. This means that we need to assume that the quantities
\[\dfrac{ \frac{d}{dt} (\e_{-} \mu_{-})}{(\e_{-} \mu_{-})^{3/2}}, \qquad \dfrac{ \frac{d}{dt} (\e_{+} \mu_{+})}{(\e_{+} \mu_{+})^{3/2}}\]
are integrable in time. 
Notice that if $\e_\pm,\mu_\pm$ are bounded below by a positive constant, then the integrability follows from \hyperlink{H2}{H2}.
Clearly this condition is trivially satisfied if $\varepsilon_{\pm}$ and $\mu_{\pm}$ are positive constants in time.



Let us make the ansatz for the incident field
\[\E_i(x,t)= A_i e^{i\omega_1 \left(\frac{k_i \cdot x}{v_-(t)}-t\right)}, \qquad \text{for $t<t_0$,}\]
where the amplitude $A_i$ is 3-d non-zero vector with complex components, $\o_1>0$, and $k_i$ is a unit vector.
Suppose that the transmitted wave has a part with direction $k_t$ and a part with direction $k_r$, given as follows.
For the transmitted field, we make the ansatz:
\[\E_t(x,t) = A_t e^{i\omega_3 \left(\frac{k_t \cdot x}{v_+(t)} - t\right)}, \qquad \text{for $t > t_0$};
\]
and for the reflected field:
\[\E_r(x,t) = A_re^{i\omega_2 \left(\frac{k_r \cdot x}{v_+(t)} - t\right)}, \qquad \text{for $t > t_0$};
\]
where, once again, $A_r, A_t$ are non-zero complex vectors and $k_t, k_r$ are unit vectors. 

Here $\o_2$ and $\o_3$ are real numbers different from zero and we will show that their signs will determine the directions of the wave vectors. 
If $\E_-=\E_i$ and $\E_+=\E_t+\E_r$, then the jump of these fields at $t=t_0$ is given by 
\[
[[\E(x,t_0)]]=\lim_{t\to t_0^+}\(\E_t(x,t)+\E_r(x,t)\)-\lim_{t\to t_0^-}\E_i(x,t).
\]
\begin{remark}
Compare this idea with the case of a spatial interface. In such a situation, since the incident wave and reflected wave are on the same side of the interface, one instead takes $\E_-=\E_i+\E_r$ and $\E_+=\E_t$. See for instance \cite{gutierrez2013refraction}.
\end{remark}
Now, if the field
\[
\E=
\begin{cases}
\E_+ & \text{for $t>t_0$}\\
\E_- & \text{for $t<t_0$}
\end{cases}
\]
satisfies the Maxwell equation \eqref{eq:maxwell equation Faraday} in the sense of distributions and the form of the fields $\E_i,\E_r$, and $\E_t$ implies that they satisfy \eqref{eq:G is C1}, \eqref{eq:limits are finite}, and \eqref{eq:derivatives if G wrt t are integrable}, it follows that the boundary condition \eqref{eq:jump of epsilon E is zero} is applicable, and we get 
\begin{align*}
0&=[[\epsilon(t_0)\E(x,t_0)]]
=
\e_+(t_0)\E_+(x,t_0)-\e_-(t_0)\E_-(x,t_0)\\
&=
\e_+(t_0)A_t e^{i\omega_3 \left(\frac{k_t\cdot x}{v_+(t_0)}-t_0\right)}
+
\e_+(t_0)A_re^{i\omega_2 \left(\frac{k_r\cdot x}{v_+(t_0)}-t_0\right)}
-
\e_-(t_0)A_i e^{i\omega_1 \left(\frac{k_i \cdot x}{v_-(t_0)}-t_0\right)}
\end{align*}
for all $x=(x_1,x_2,x_3)$.

If we define
\[
m_i = \omega_1 k_i/ v_-(t_0), \qquad m_r = \omega_2 k_r/v_+(t_0), \qquad m_t=\omega_3 k_t/v_+(t_0),
\]
with $v_-(t_0)=\lim_{t\to t_0^-}v_-(t)$ and $v_+(t_0)=\lim_{t\to t_0^+}v_+(t)$,
then we get an equation of the form
\begin{align}\label{exponential_equation}
B_t e^{im_t\cdot x}+B_r e^{i m_r\cdot x}-B_i e^{i m_i \cdot x}=0\quad \text{for all $x=(x_1,x_2,x_3)$,}
\end{align}
where
\[B_t = \e_+(t_0)A_t e^{-i\omega_3 t_0}, \quad B_i = \e_-(t_0)A_i e^{-i\omega_1 t_0}, \quad B_r = \e_+(t_0)A_re^{-i \omega_2 t_0}.\]
Now, from Lemma \ref{lm:exponentials}, equation \eqref{exponential_equation} implies, that $m_i = m_t=m_r$, i.e.,
\begin{equation}\label{eq:exponents are all equal}
\dfrac{\omega_1 k_i}{v_-(t_0)}=\dfrac{\omega_2 k_r}{v_+(t_0)}=\dfrac{\omega_3 k_t}{v_+(t_0)}.
\end{equation}
This gives the following proposition showing general relationships between the wave vectors $k_r$, $k_t$, and $k_i$.
\begin{prop}\label{prop:snell law}
Under the previous assumptions, we have
\begin{align}\label{kt_kr_ki_relation}
\begin{split}
&k_t  =\dfrac{\omega_1}{\omega_3}\dfrac{v_+(t_0)}{v_-(t_0)} k_i,\\
&k_r =\dfrac{\omega_1}{\omega_2} \dfrac{v_+(t_0)}{v_-(t_0)}k_i
\end{split}.
\end{align}
\end{prop}
Since the wave vectors are all unit, taking absolute values in \eqref{kt_kr_ki_relation} yields 
\[
\left| \dfrac{\omega_1}{\omega_3}\dfrac{v_+(t_0)}{v_-(t_0)}\right|=\left|\dfrac{\omega_1}{\omega_2} \dfrac{v_+(t_0)}{v_-(t_0)}\right|=1
\]
which implies 
\[
\dfrac{\omega_1}{\omega_3}=\pm \dfrac{v_-(t_0)}{v_+(t_0)},\quad \dfrac{\omega_1}{\omega_2}=\pm \dfrac{v_-(t_0)}{v_+(t_0)}.
\]
We have assumed at the outset that $\o_1>0$. So, if $\o_3>0$, then $k_t=k_i$, which is consistent with \cite[Equation (5)]{galiffi2022photonics}. Moreover, the fact that $\omega_3=\dfrac{v_+(t_0)}{v_-(t_0)} \omega_1$ is precisely the Snell's Law for the transmitted wave in \cite[Equations (13)-(14)]{mendoncca2002time}. 
On the other hand, if $\o_3<0$ we then get $k_t=-k_i$.
A similar analysis with $\o_2$ yields that $k_r=k_i$ if $\o_2>0$, and $k_r=-k_i$ if $\o_2<0$. This agrees with  \cite[Equation (4)]{mostafa2024temporal} and \cite[Equation (13)]{mendoncca2002time}.


\begin{remark}
As mentioned above, the equations (\ref{kt_kr_ki_relation}) can be seen as a generalized Snell's Law at a temporal interface which does not require that the velocities on either side of the  interface be constant. Indeed, this can be written as
\begin{align}\label{general_Snell_transmitted}
\omega_3 n_2(t_0) k_t =\omega_1 n_1(t_0)k_i
\end{align}
and
\begin{align}\label{general_Snell_reflected}
\omega_2 n_2(t_0) k_r =\omega_1 n_1(t_0)k_i
\end{align}
In the plane $(x, ct)$, one can define an angle $\alpha$ similar to the incidence angle $\theta$ in the $(x,y)$ plane with spatial interfaces, such that $\tan(\alpha_j) = \dfrac{1}{n_j}$ for $j=1,2,3$. This angle can be called the angle of temporal incidence \cite{mendoncca2000theory}. Then, from (\ref{general_Snell_transmitted}) and (\ref{general_Snell_reflected}) we obtain the scalar law
\[| \omega_2|\tan(\alpha_1) = \omega_1 \tan(\alpha_2) \]
which agrees with the scalar law in equation (15) in \cite{mendoncca2002time}. Since $\tan(x)$ can take on any real value, we see that there is no notion of total internal reflection as in the case of classical spatial interfaces. That is, there is no wave that propagates backwards in time. 
Finally, the material parameters $\varepsilon$ and $\mu$  should be differentiable in time away from the interface $t=t_0$ and could in principle also vary in space. However, with spatially varying velocities, it is not clear if Lemma \ref{lm:exponentials} applies, so we have not considered this case. 
\end{remark}

\subsection{Case when $\e$ and $\mu$ are jump functions and calculation of the amplitudes}\label{subsec:epsilon mu jump constants}

Now consider the case when 
\[
\e(t)=
\begin{cases}
\e_- &\text{for $a<t<t_0$}\\
\e_+ &\text{for $t_0<t<b$}
\end{cases},
\qquad 
\mu(t)=
\begin{cases}
\mu_- &\text{for $a<t<t_0$}\\
\mu_+ &\text{for $t_0<t<b$}
\end{cases}
\]
with $\mu_{\pm},\e_{\pm}$ positive constants.
So $v_-(t)=v_-$ for $a<t<t_0$ and $v_+(t)=v_+$ for $t_0<t<b$ are both constant. This is a common application in practice, see e.g. \cite{galiffi2022photonics} and the references therein.
This form of velocities also enables us to compute the associated magnetic fields and establish the boundary conditions for them. Together with the boundary conditions already obtained for the electric fields, we will be able to calculate the amplitudes of the transmitted and reflected waves.

In fact, let us calculate first the magnetic fields. Consider the following Maxwell equation \eqref{Faraday law} for the incident field when $t<t_0$: seek $\mathbf{H}_i$ satisfying
\[\nabla \times \E_i = -\dfrac{\mu_-}{c} \dfrac{\partial \mathbf{H}_i}{\partial t}.\]
Since
\[\E_i(x,t)= A_i e^{i\omega_1 \left(\frac{k_i \cdot x}{v_-}-t\right)}\]
we see that
\[\nabla \times \E_i =-i\omega_1 \left( A_i \times \dfrac{k_i}{v_-}\right) e^{i\omega_1 \left( \frac{k_i\cdot x}{v_-}-t\right)}\]
and so integrating in time yields
\begin{align}\label{Hi_calculated}
\mathbf{H}_i = -\dfrac{c}{\mu_-}\int \nabla \times \E_i \,dt=-\dfrac{c}{\mu_-}\E_i \times \dfrac{k_i}{v_-}
\end{align}
plus a field depending only on $x$ which is assumed to be zero.
Now suppose $t>t_0$. 
Since
\[\E_r(x,t)=A_re^{i\omega_2 \left(\frac{k_r\cdot x}{v_+}-t\right)}\]
we similarly find that
\begin{align}\label{Hr_calculated}
\mathbf{H}_r = -\dfrac{c}{\mu_+} \E_r \times \dfrac{k_r}{v_+}
\end{align}
plus a field depending only on $x$ which is also assumed to be zero.
Finally, since
\[\E_t(x,t) = A_t e^{i\omega_3 \left(\frac{k_t\cdot x}{v_+}-t\right)}\]
we find that
\begin{align}\label{Ht_calculated}
\mathbf{H}_t = -\dfrac{c}{\mu_+} \E_t \times \dfrac{k_t}{v_+}.
\end{align}
Now we may consider the magnetic boundary condition (\ref{eq:jump of mu H is zero}). The jump of $\mu \mathbf{H}$ by definition is given by 
\begin{align}\label{muH-written-out}
\begin{split}
[[\mu(x,t_0) \mathbf{H}(x,t_0)]] &=\mu_+\mathbf{H}_+(t_0)-\mu_-\mathbf{H}_-(t_0)\\
&=\mu_+\left( -\dfrac{c}{\mu_+} \E_t \times \dfrac{k_t}{v_+} -\dfrac{c}{\mu_+} \E_r \times \dfrac{k_r}{v_+}\right)-\mu_-\left( -\dfrac{c}{\mu_-}\E_i \times \dfrac{k_i}{v_-}\right)\\ 
&=-c \left(A_t\times \dfrac{k_t}{v_+}e^{i\omega_3 \left(\frac{k_t\cdot x}{v_+}-t\right)} +A_r \times \dfrac{k_r}{v_+}e^{i\omega_2 \left(\frac{k_r\cdot x}{v_+}-t\right)}\right)
+c\left( A_i \times \dfrac{k_i}{v_-}e^{i\omega_1 \left(\frac{k_i \cdot x}{v_-}-t\right)} \right), \quad \forall \; x
\end{split}
\end{align}
and at $t=t_0$. Here $\H_+=\H_r+\H_t$.
The boundary condition (\ref{eq:jump of mu H is zero}) then implies that
\[-c \left(A_t\times \dfrac{k_t}{v_+}e^{i\omega_3 \left(\frac{k_t\cdot x}{v_+}-t_0\right)} +A_r \times \dfrac{k_r}{v_+}e^{i\omega_2 \left(\frac{k_r\cdot x}{v_+}-t_0\right)}\right)
+c\left( A_i \times \dfrac{k_i}{v_-}e^{i\omega_1 \left(\frac{k_i \cdot x}{v_-}-t_0\right)} \right)=0\]
for all $x$.

But since we know from \eqref{eq:exponents are all equal} that the exponentials in $x$ are all equal
we obtain
\begin{align}\label{eqn_to_solve_magnetic_BC}
-A_te^{-i\omega_3 t_0} \times \dfrac{k_t}{v_+} -A_re^{-i\omega_2 t_0}\times\dfrac{k_r}{v_+}+A_ie^{-i\omega_1 t_0}\times \dfrac{k_i}{v_-} =0, 
\end{align}
that is,
\begin{align*}
-A_te^{-i\omega_3 t_0} \times k_i \,\dfrac{\omega_1}{v_-\omega_3}-A_re^{-i\omega_2 t_0}\times k_i\, \dfrac{\omega_1}{v_-\omega_2} +A_ie^{-i\omega_1 t_0}\times \dfrac{k_i}{v_-} =0, 
\end{align*}
cancelling $v_-$ yields
\begin{align*}
-A_te^{-i\omega_3 t_0} \times k_i \,\dfrac{\omega_1}{\omega_3}  -A_re^{-i\omega_2 t_0}\times k_i\, \dfrac{\omega_1}{\omega_2}+A_ie^{-i\omega_1 t_0}\times k_i
=0, 
\end{align*}
so
\begin{align*}
\(-\dfrac{\omega_1}{\omega_3} A_te^{-i\omega_3 t_0}-\dfrac{\omega_1}{\omega_2} A_re^{-i\omega_2 t_0} +A_ie^{-i\omega_1 t_0}
 \)\times k_i =0. 
\end{align*}
From \eqref{eq:exponents are all equal} and letting $x=0$ in \eqref{exponential_equation} we get
\[\e_+A_t e^{-i\omega_3 t_0} 
+\e_+A_re^{-i \omega_2 t_0}
-\e_-A_i e^{-i\omega_1 t_0}
=0
.\]


If we set 
\[
B_i=A_ie^{-i\omega_1 t_0},\quad B_r=A_re^{-i\omega_2 t_0},\quad B_t=A_te^{-i\omega_3 t_0},\quad k_i=(k_1^i,k_2^i,k_3^i)\]
\[
B_i=(B_1^i,B_2^i,B_3^i),\quad B_r=(B_1^r,B_2^r,B_3^r),\quad B_t=(B_1^t,B_2^t,B_3^t),
\]
then we need to solve in $B_r$ and $B_t$ the system of equations
\begin{align}
\e_+ \,B_t 
+\e_+ \, B_r
-\e_-\,B_i&=0\label{eq:equation for B in epsilon}\\
\(-\frac{\omega_1}{\omega_3} \,B_t -\frac{\omega_1}{\omega_2}\, B_r+B_i
 \)\times k_i &=0\label{eq:equation for B cross ki}
\end{align}
which can be written
\begin{align*}
 B_r+ B_t&=(\e_-/\e_+)\,B_i\\
\(\frac{\omega_1}{\omega_2}\, B_r+\frac{\omega_1}{\omega_3} \,B_t\)\times k_i
  &=B_i\times k_i.
\end{align*}
We now proceed to find $B_r$ and $B_t$ from this system of equations. 
We shall prove the following.
\begin{prop}\label{prop:formulas for the amplitudes}
If $\o_2\neq \o_3$, then 
\begin{equation}\label{eq:amplitude Br in terms of Bi}
B_r=\dfrac{1-(\omega_1/\omega_3) (\e_-/\e_+)}{(\omega_1/\omega_2)-(\omega_1/\omega_3)} B_i,
\end{equation}
and 
\begin{equation}\label{eq:amplitude Bt in terms of Bi}
B_t=
\dfrac{(\omega_1/\omega_2)(\e_-/\e_+)-1}{(\omega_1/\omega_2)-(\omega_1/\omega_3)}B_i.
\end{equation}
If $\o_2=\o_3$, then 
\[\(A_t+A_r\)e^{-i\o_2 t_0}=(\e_-/\e_+)B_i.\]
\end{prop}
\begin{proof}
The augmented matrix of the system is
\[
M=\begin{bmatrix}
 1 & 0 & 0 &  1 & 0 & 0 & B_1^i \e_-/\e_+ \\
 0 & 1 & 0 & 0 & 1 & 0 & B_2^i \e_-/\e_+ \\
 0 & 0 & 1 & 0 & 0 & 1 & B_3^i \e_-/\e_+ \\
 0 & \frac{k_3^i \omega_1}{\omega_2} & -\frac{k_2^i \omega_1}{\omega_2} & 0 & \frac{k_3^i \omega_1}{\omega_3} & -\frac{k_2^i \omega_1}{\omega_3} & B_2^i k_3^i-B_3^i k_2^i \\
 -\frac{k_3^i \omega_1}{\omega_2} & 0 & \frac{k_1^i \omega_1}{\omega_2} & -\frac{k_3^i \omega_1 }{\omega_3} & 0 & \frac{k_1^i \omega_1 }{\omega_3} & B_3^i k_1^i- B_1^i k_3^i \\
 \frac{k_2^i \omega_1}{\omega_2} & -\frac{k_1^i \omega_1 }{\omega_2 } & 0 & \frac{k_2^i \omega_1 }{\omega_3} & -\frac{k_1^i \omega_1 }{\omega_3} & 0 & B_1^i k_2^i-B_2^i k_1^i 
\end{bmatrix}.
\]
If we let  
\[
N=\begin{bmatrix}
0 & k_3^i  & -k_2^i \\
-k_3^i & 0 & k_1^i \\
k_2^i & -k_1^i & 0
\end{bmatrix},
\]
(notice $\det N=0$) then the matrix $M$ without the last column equals
\[
\begin{bmatrix}
 Id &  Id\\
(\omega_1/\omega_2) N & (\omega_1/\omega_3) N
\end{bmatrix}
\]
The system of equations can then be written as
\[
\begin{bmatrix}
 Id &  Id\\
(\omega_1/\omega_2) N & (\omega_1/\omega_3) N
\end{bmatrix} 
\begin{bmatrix}
B_r\\
B_t
\end{bmatrix}
=
\begin{bmatrix}
(\e_-/\e_+) \,B_i\\
B_i\times k_i
\end{bmatrix},
\]
where the unknowns are $B_r,B_t$.
This means
\[
B_r+B_t=(\e_-/\e_+) \,B_i,\qquad N\((\omega_1/\omega_2) B_r+(\omega_1/\omega_3) B_t\)=\(B_i\times k_i\).
\]
From the first equation $B_t=(\e_-/\e_+) \,B_i-B_r$ and substituting in the second equation yields
\[
N\((\omega_1/\omega_2) B_r+(\omega_1/\omega_3) \((\e_-/\e_+) \,B_i-B_r\)\)=\(B_i\times k_i\)
\]
which gives
\[
N\(\((\omega_1/\omega_2)-(\omega_1/\omega_3)\) B_r\)
=\(B_i\times k_i\)-N(\omega_1/\omega_3) (\e_-/\e_+) \,B_i
=
\(1-(\omega_1/\omega_3) (\e_-/\e_+)\)\(B_i\times k_i\).
\]
{\bf Case when $\omega_2\neq \omega_3$.} Then $B_r$ must solve 
\[
N\,B_r=\dfrac{1-(\omega_1/\omega_3) (\e_-/\e_+)}{(\omega_1/\omega_2)-(\omega_1/\omega_3)}\(B_i\times k_i\)\]
which means that to have a solution $B_r$ the vector $\dfrac{1-(\omega_1/\omega_3) (\e_-/\e_+)}{(\omega_1/\omega_2)-(\omega_1/\omega_3)}\(B_i\times k_i\)$ must be in the image of $N$.
Given a 3-d vector $X$ we have $NX=X\times k_i$.
So the vector $B_r$ must satisfy the equation 
\[
\(B_r-\dfrac{1-(\omega_1/\omega_3) (\e_-/\e_+)}{(\omega_1/\omega_2)-(\omega_1/\omega_3)} B_i\)\times k_i=0
\]
which means the vector $B_r-\dfrac{1-(\omega_1/\omega_3) (\e_-/\e_+)}{(\omega_1/\omega_2)-(\omega_1/\omega_3)} B_i$ must be parallel to the vector $k_i$, that is,
\[
B_r-\dfrac{1-(\omega_1/\omega_3) (\e_-/\e_+)}{(\omega_1/\omega_2)-(\omega_1/\omega_3)} B_i=\lambda \,k_i
\]
but since $\omega_2\neq \omega_3$, from \eqref{compatibility_condition_Eplus_omega2neqomega3} below and \eqref{kt_kr_ki_relation}, dotting the last equation with $k_i$ yields $\lambda=0$ and we then obtain that 
\eqref{eq:amplitude Br in terms of Bi}.
We also obtain from this that 
\begin{equation*}
B_t=(\e_-/\e_+)B_i-B_r=\((\e_-/\e_+)-\dfrac{1-(\omega_1/\omega_3) (\e_-/\e_+)}{(\omega_1/\omega_2)-(\omega_1/\omega_3)}\)B_i
=
\dfrac{(\omega_1/\omega_2)(\e_-/\e_+)-1}{(\omega_1/\omega_2)-(\omega_1/\omega_3)}B_i,
\end{equation*}
which is \eqref{eq:amplitude Bt in terms of Bi}.

{\bf Case when $\omega_2=\omega_3$.} We have 
\[
N\( B_r+ B_t\)=(\omega_2/\omega_1)\(B_i\times k_i\)
\]
from \eqref{eq:equation for B cross ki}, and since $B_r+B_t=(\e_-/\e_+)B_i$, if there is a solution $B_r,B_t$ then the vector $B_i$ must satisfy 
\[
N\((\e_-/\e_+)B_i\)=(\e_-/\e_+) \(B_i\times k_i\)=(\omega_2/\omega_1)\(B_i\times k_i\)
\]
which is equivalent to
\[
\e_-\omega_1 =\e_+ \omega_2.
\]
In this case \eqref{eq:equation for B in epsilon} and \eqref{eq:equation for B cross ki} are the same equation, and so 
we get $B_r+B_t=(\e_-/\e_+)B_i$ which has infinitely many solutions $B_r,B_t$.
In this case from \eqref{kt_kr_ki_relation} we have $k_r=k_t$ and so the electric field 
\begin{align*}
\E_+&=A_t e^{i\omega_3 \left(\frac{k_t\cdot x}{v_+(t)}-t\right)}+A_re^{i\omega_2 \left(\frac{k_r\cdot x}{v_+(t)}-t\right)}
=\(A_t+A_r\) e^{i\omega_3 \left(\frac{k_t\cdot x}{v_+(t)}-t\right)}
\end{align*}
showing that there are not distinctive transmitted and reflected waves. 
This means the amplitudes of the transmitted and reflected waves satisfy 
\[\(A_t+A_r\)e^{-i\o_2 t_0}=(\e_-/\e_+)B_i.\]
Notice \eqref{compatibility_condition_Eplus_omega2equalomega3} must hold in this case, so the amplitudes $A_t, A_r$ cannot be arbitrary.
\end{proof}

Also note that when $\omega_2=\omega_3$, we have that $k_t=k_r$, but the relation to $k_i$ depends on the sign of the frequencies. Indeed, if $\omega_3>0$ then $k_t=k_r=k_i$, but if $\omega_3<0$ then we have $k_t=k_r =-k_i$. 

\subsubsection{Reflection and transmission coefficients}\label{subsec:reflection and transmission coefficients}
Define the reflection coefficient $\mathcal{R}$ as $|B_r|/|B_i|$, we see that, if $\omega_2=-\omega_3$ and if $\omega_3>0$, then from \eqref{eq:amplitude Br in terms of Bi} we have that 
\begin{align*}
B_r&=\dfrac12 \(\dfrac{\e_-}{\e_+}-\dfrac{\o_3}{\o_1}\)B_i=\dfrac12 \(\dfrac{\e_-}{\e_+}-\dfrac{\sqrt{\e_-\mu_-}}{\sqrt{\e_+\mu_+}}\)B_i
\end{align*}
since $\dfrac{\omega_1}{\omega_3}=\pm \dfrac{v_-}{v_+},\quad \dfrac{\omega_1}{\omega_2}=\pm \dfrac{v_-}{v_+}$, and $\o_1>0$. 
Since $B_t=(\e_-/\e_+)B_i-B_r$ we get in this case that
\[
B_t=\dfrac{\e_-}{\e_+}B_i-\dfrac12 \(\dfrac{\e_-}{\e_+}-\dfrac{\sqrt{\e_-\mu_-}}{\sqrt{\e_+\mu_+}}\)B_i
=
\dfrac12 \(\dfrac{\e_-}{\e_+}+\dfrac{\sqrt{\e_-\mu_-}}{\sqrt{\e_+\mu_+}}\)B_i.
\]
That is, we have shown that the reflection coefficient $\mathcal{R}$ is given by
\begin{align}\label{reflection_coefficient_corrected}
\mathcal{R}=\left| \dfrac{B_r}{B_i}\right|=\dfrac{1}{2}\left|\dfrac{\e_-}{\e_+}-\dfrac{\sqrt{\e_-\mu_-}}{\sqrt{\e_+\mu_+}}\right|
=\dfrac12 \,
\sqrt{\dfrac{\e_-}{\e_+}}
\left|
\sqrt{\dfrac{\e_-}{\e_+}}-\sqrt{\dfrac{\mu_-}{\mu_+}}
\right|
\end{align}
which agrees with \cite[Equation 4]{xiao2014reflection}.
Note that the quantity inside the absolute value can be positive or negative depending on the relationship between the impedances.
We also see that the transmission coefficient $\mathcal{T}$ is given by
\begin{align}\label{transmission_coefficient}
\mathcal{T}=\left| \dfrac{B_t}{B_i}\right| = \dfrac{1}{2}\left|\dfrac{\e_-}{\e_+}+\dfrac{\sqrt{\e_-\mu_-}}{\sqrt{\e_+\mu_+}}\right| =\dfrac{1}{2}\left(\dfrac{\e_-}{\e_+}+\dfrac{\sqrt{\e_-\mu_-}}{\sqrt{\e_+\mu_+}}\right).
\end{align}
This agrees with \cite[Equation 5]{xiao2014reflection}. Note that we do not have conservation of energy in this case, but rather we see that $\mathcal{R}+\mathcal{T}$ depends on the impedances on either side of the interface. Indeed, let $Z_1=\sqrt{\dfrac{\mu_-}{\e_-}}$ and $Z_2=\sqrt{\dfrac{\mu_+}{\e_+}}$. If $Z_1<Z_2$ then we have
\begin{align}\label{conservation_of_energy1}
\mathcal{R}+\mathcal{T}=\dfrac{\e_-}{\e_+}.
\end{align}
If $Z_1>Z_2$ then we have
\begin{align}\label{conservation_of_energy2}
\mathcal{R}+\mathcal{T}=\sqrt{\dfrac{\e_- \mu_-}{\e_+ \mu_+}}=\dfrac{n_+}{n_-}.
\end{align}


Now, if $\omega_2 =-\omega_3$ and $\omega_3<0$ while $\omega_1>0$, then we find that
\[ \mathcal{R}= \dfrac{1}{2}\(\dfrac{\e_-}{\e_+}+\dfrac{\sqrt{\e_-\mu_-}}{\sqrt{\e_+\mu_+}}\)\]
and
\[\mathcal{T}=\dfrac{1}{2}\left|\dfrac{\e_-}{\e_+}-\dfrac{\sqrt{\e_-\mu_-}}{\sqrt{\e_+\mu_+}}\right|\]
That is, the transmission and reflection coefficients have switched roles in this case. 

Finally, note that negative time refraction is also possible \cite{bruno2020broad, vezzoli2018optical, pendry2008time}. So now suppose that $n_->0$ but $n_+<0$ with $\e_+, \mu_+ < 0$, so that $n_+ =-\sqrt{\e_+ \mu_+}$. Then in particular we have $v_+<0$. The previous analysis up to equation (\ref{eq:exponents are all equal}) does not see the sign of these material parameters so this result still holds. Now, (\ref{kt_kr_ki_relation}) then gives
\[k_t  =\dfrac{\omega_1}{\omega_3}\dfrac{v_+}{v_-} k_i\]
and so, again after taking absolute values, we obtain if $\omega_1, \omega_3$ have the same sign then
\begin{align}\label{neg-time-refraction}
k_t=-k_i
\end{align}
because in this case now $\omega_3=-(v_+/v_-)\omega_1$.  Similarly, if $\omega_1, \omega_3$ have different signs, then we obtain $k_t = k_i$.
Similarly, if $\omega_2<0$ (i.e. $\omega_1, \omega_2$ have different signs), then again from (\ref{eq:exponents are all equal}) we find that $\omega_2=(v_+/v_-)\omega_1$ and hence
\begin{align}\label{neg-time-refraction-2}
k_r = k_i
\end{align}

\subsection{Divergence free conditions}\label{sec:divergence}
Recall that we have defined
\[\E_+ = \E_t+\E_r, \qquad \E_- = \E_i\]
Assuming zero charge density, we also have the following divergence equations from the Maxwell system:
\begin{align}\label{diveqn1}
\nabla \cdot \left(\varepsilon_{\pm} \E_{\pm}\right) = 0
\end{align}
and
\begin{align}\label{diveqn2}
\nabla \cdot \left(\mu_{\pm} \mathbf{H}_{\pm}\right)=0
\end{align}
where recall that we associate $\e_+, \mu_+$ with $\E_+$ and $\e_-, \mu_-$ with $\E_-$. Again let us assume that each $\varepsilon_{\pm}, \mu_{\pm}$ are positive constants.
First consider $\E_-$. In light of (\ref{diveqn1}) we find that
\begin{align}\label{compatability_condition_Eminus}
A_i \cdot \dfrac{\omega_1}{v_-} k_i =0
\end{align}
We also need to verify (\ref{diveqn2}) with $\mathbf{H}_i$ given by (\ref{Hi_calculated}). Note that we find
\begin{align*}
\nabla \times \E_i&=  e^{i\omega_1\Phi_i} \nabla \times A_i+ \left( \nabla e^{i\omega_1\Phi_i}\right)\times A_i\\
&= e^{i\omega_1 \Phi_i} \nabla\times A_i+i\omega_1 \dfrac{k_i}{v_-} e^{i\omega_1 \Phi_i} \times A_i\\
&=i\omega_1 \dfrac{k_i}{v_-} e^{i\omega_1 \Phi_i} \times A_i
\end{align*}
where
\[\Phi_i \coloneqq \dfrac{k_i \cdot x}{v_-}-t\]
We find that
\begin{align*}
\nabla \cdot (\mu_- \mathbf{H}_i) &= \nabla \cdot \left(-c \E_i \times \dfrac{k_i}{v_-}\right)\\
&=\dfrac{k_i}{v_-}\cdot (\nabla \times -c\E_i)+c\E_i\cdot (\nabla \times \dfrac{k_i}{v_-})
\end{align*}
But by the previous calculation of $\nabla \times \E_i$ above, since $k_i/v_-$ appears twice in the term $k_i/v_- \cdot \nabla \times (-c\E_i)$, we find that $\nabla\cdot (\mu_- \mathbf{H}_i)=0$. Hence (\ref{diveqn2}) is satisfied for $\mathbf{H}_-$.

Consider now the total field
\[\E_+(t) = A_t e^{i\omega_3 \left(\frac{k_t\cdot x}{v_+}-t\right)}+A_r e^{i\omega_2 \left(\frac{k_r\cdot x}{v_+}-t\right)}\]
In light of (\ref{diveqn1}) we find that we must have
\begin{align*}
\left(A_t \cdot \dfrac{\omega_3}{v_+} k_t \right)e^{-i\omega_3 t} +\left( A_r \cdot \dfrac{\omega_2}{v_+} k_r \right)e^{-i\omega_2 t}=0, \qquad \forall \; t > t_0
\end{align*}
Now, if $\omega_2\neq \omega_3$, this yields
\begin{align}\label{compatibility_condition_Eplus_omega2neqomega3}
A_t \cdot \dfrac{\omega_3}{v_+} k_t =0 \; \qquad \text{ and } \qquad A_r \cdot \dfrac{\omega_2}{v_+} k_r=0
\end{align}
If $\omega_2=\omega_3$ we have
\begin{align}\label{compatibility_condition_Eplus_omega2equalomega3}
\left( A_t +A_r\right) \cdot \dfrac{\omega_3}{v_+} k_t = \left( A_t +A_r\right) \cdot \dfrac{\omega_3}{v_+} k_r=0
\end{align}

We also need to verify (\ref{diveqn2}) with $\mathbf{H}_+$. Let us compute $\mathbf{H}_+$. 
Since
\[\E_+(t)= A_t e^{i\omega_3 \left(\frac{k_t \cdot x}{v_+}-t\right)}+A_r e^{i\omega_2 \left(\frac{k_r \cdot x}{v_+}-t\right)}\]
we see that
\[\nabla \times \E_+ =-i\omega_3 \left( A_t \times \dfrac{k_t}{v_+}\right) e^{i\omega_3 \left( \frac{k_t\cdot x}{v_+}-t\right)} -i\omega_2 \left( A_r \times \dfrac{k_r}{v_+}\right) e^{i\omega_2 \left( \frac{k_r\cdot x}{v_+}-t\right)}
\]
and so integrating in time yields
\begin{align}\label{Hi_calculated}
\mathbf{H}_+ = -\dfrac{c}{\mu_+}\int \nabla \times \E_+ dt=-\dfrac{c}{\mu_+}\E_t \times \dfrac{k_t}{v_+}-\dfrac{c}{\mu_+}\E_r\times \dfrac{k_r}{v_+}
\end{align}
plus a field depending only on $x$ which is assumed to be zero.


Similar to above, we find that
\begin{align*}
\nabla \cdot (\mu_+ \mathbf{H}_+) &= \nabla \cdot \left(-c \E_t \times \dfrac{k_t}{v_+}\right)+\nabla \cdot  \left(-c \E_r \times \dfrac{k_r}{v_+}\right)
\end{align*}
and both cross products will end up being zero. Hence (\ref{diveqn2}) is satisfied with $\mathbf{H}_+$ as well.



\subsection{Exponential Lemma}\label{sec:exponentials}
In our previous analysis, we required the following exponential lemma:

\begin{lemma}\label{lm:exponentials}
Let $A_1, \cdots, A_N \in \C^n\setminus \{0\}$, and $\o_1, \cdots, \o_N \in \R$.
If 
\begin{equation}\label{eq:sum exponentials equal zero}
\sum_{j=1}^N A_j\,e^{i\,\o_j x} = 0 \qquad \forall x \in \R,
\end{equation}
 then $\o_1=\cdots =\o_N$.
 \end{lemma}
\begin{proof}

Suppose first that $n=1$. Differentiating \eqref{eq:sum exponentials equal zero} $k$ times with respect to $x$ yields
$
\sum_{j=1}^N (i\,\o_j)^k\,A_j\,e^{i\,\o_j x} = 0$.
If we let
\[
B = B(x) =
\begin{pmatrix}
e^{i\,\o_1 x} & e^{i\,\o_2 x} & \cdots & e^{i\,\o_N x} \\
i\,\o_1\,e^{i\,\o_1 x} & i\,\o_2\,e^{i\,\o_2 x} & \cdots & i\,\o_N\,e^{i\,\o_N x} \\
(i\,\o_1)^2\,e^{i\,\o_1 x} & (i\,\o_2)^2\,e^{i\,\o_2 x} & \cdots & (i\,\o_N)^2\,e^{i\,\o_N x} \\
\vdots & \vdots & \ddots & \vdots \\
(i\,\o_1)^{N-1}\,e^{i\,\o_1 x} & (i\,\o_2)^{N-1}\,e^{i\,\o_2 x} & \cdots & (i\,\o_N)^{N-1}\,e^{i\,\o_N x}
\end{pmatrix},
\]
then the vector $A=\(A_1,\cdots ,A_N\)\neq 0$ satisfies the system of equations $BA^t=0$
and so the determinant 
\begin{align*}
\det B &= e^{i\,\o_1 x} \cdots e^{i\,\o_N x}
\det 
\begin{pmatrix}
1 & 1 & \cdots & 1 \\
i\,\o_1 & i\,\o_2 & \cdots & i\,\o_N \\
(i\,\o_1)^2 & (i\,\o_2)^2 & \cdots & (i\,\o_N)^2 \\
\vdots & \vdots & \ddots & \vdots \\
(i\,\o_1)^{N-1} & (i\,\o_2)^{N-1} & \cdots & (i\,\o_N)^{N-1}
\end{pmatrix} \\
&= e^{i\,\o_1 x} \cdots e^{i\,\o_N x}\,
\prod_{1 \leq k < \ell \leq N} (i\o_\ell - i\o_k)=0
\end{align*}
by the Vandermonde determinant formula. Therefore all $\o_j$ are equal, and so the lemma follows for $n=1$.

Suppose next that $n>1$. Write $A_j=\(a_1^j,\cdots ,a_n^j\)$, $1\leq j\leq N$. So \eqref{eq:sum exponentials equal zero} implies
\[
\sum_{j=1}^N a_k^j\,e^{i\,\o_j x} = 0 \qquad \forall x \in \R
\]
for $1\leq k\leq n$.
Suppose, by contradiction, that not all $\omega_j$ are equal.
Then there exists $1 < m \leq N$ such that relabeling $\omega_j$ we can write $\omega_1 < \omega_2 < \cdots < \omega_m$ and $\omega_j = \omega_m$ for $m \leq j \leq N$. Hence we can write
\[
\sum_{j=1}^{m-1} A_j \, e^{i \, \omega_j x} + \left(\sum_{j=m}^N A_j\right) \, e^{i \, \omega_m x} = 0
\]
for all $x \in \mathbb{R}$, which written in components means
\[
\sum_{j=1}^{m-1} a_k^j \, e^{i \, \omega_j x} + \left(\sum_{j=m}^N a_k^j\right) \, e^{i \, \omega_m x} = 0,
\]
for $1\leq k \leq n$.  
From the case when $n=1$, we obtain
\[
a_k^j=0 \quad\text{for $1\leq j\leq m-1$ and }\sum_{j=m}^N a_k^j=0\]
for $1\leq k\leq n$. Consequently, the vectors $A_j=\(a_1^j,\cdots ,a_n^j\)=0$ for $1\leq j\leq m-1$, which leads to a contradiction.


\end{proof}

\setcounter{equation}{0}

\bibliographystyle{plain}
\bibliography{timevaryingoptics.bib}


\end{document}